\documentclass[12pt]{amsart}
\usepackage{geometry} 
\usepackage{amsmath}
\usepackage{mathtools}
\usepackage{graphicx}
\usepackage{amssymb}
\usepackage{enumerate}
\usepackage{epstopdf}
\newcommand\K{\kappa}
\newtheorem{proposition}{Proposition}
\newtheorem{example}{Example}
\begin{document}

\title{A study of a debt-influenced equilibrium of the Keen model
}


\author{Son Van         \and
        Teng Zhang 
}




\maketitle

\begin{abstract}
The Keen model is a mathematical model that describes the dynamic evolution of wages, employment, and debt based on the known Minsky's Financial Instability Hypothesis. It consists of three first order nonlinear  ordinary differential equations. There exists an equilibrium state that corresponds to a collapse of both wage and employment but certain debts remain. In this paper, we construct an example in which this equilibrium appears to be stable. More interestingly, the debt becomes negative, which indicates some source of income without laboring.

\end{abstract}

\section{Introduction}
\label{intro}
\subsection{Goodwin model}
	Harrod--Domar model and Phillips curve are basic models in Macroeconomics textbooks. While the Harrod--Domar was a popular way that economists thought about economic growth, the Phillips curve is a way to look at the dynamics of wages. Goodwin model \cite{Goodwin} is a classical model that combines ideas from Harrod--Domar growth model with Phillips curve to describe the dynamics between employment rate and wage share. In the paper of Grasselli and Lima \cite{Grasselli}, extensions of the Goodwin model are studied. In this paper, we adopt notations and formulations from \cite{Grasselli}. More precisely,
	
	$\nu$ is capito-output ratio,
	
	$\alpha$ is productivity rate,
	
	$\beta$ is labor force growth rate,
	
	$\delta$ is depreciation rate,
	
	$w(t)$ is wage,
	
	$K(t)$ is capital,
	
	$L(t)$ is number of employed workers,
	
	$a(t)=a_0e^{\alpha t}$ is labor productivity,
	
	$N(t)=N_0e^{\beta t}$ is total labor force,
	
	$\lambda(t)=\frac{L(t)}{N(t)}$ is employment rate,
	
	$\omega(t)=\frac{w(t)}{a(t)}$ is wage share.
	
	According to \cite{Grasselli}, we assume that the total output (GDP) per year is given by the Leonteif production function
	\begin{equation*}
		Y(t)=\min\bigg\{\frac{K(t)}{\nu},a(t)L(t)\bigg\}
	\end{equation*}
	and the capital is fully utilized. We then have
	\begin{equation*}
		Y(t)=\frac{K(t)}{\nu}=a(t)L(t).
	\end{equation*}
	We also assume that
	\begin{equation*}
		\dot{w}=\Phi(\lambda) w
	\end{equation*}
	and 
	\begin{equation*}
		\dot{K}=(Y-wL)-\delta K = (1-\omega)Y-\delta K,
	\end{equation*}
	where $\Phi(\lambda)$ is the Phillips curve with $d{\Phi}/d\lambda>0$.
	
	Under the linearized Phillips curve assumption $$\Phi(t)=-\phi_0+\phi_1 \lambda,$$ where $\phi_0$ and  $\phi_1$ are positive constants, the following system of first order nonlinear differential equations is then derived in \cite{Grasselli}
	\begin{eqnarray*}
		\dot{\omega}&=&\omega(-\phi_0+\phi_1\lambda-\alpha)\\
		\dot{\lambda}&=&\lambda (\frac{1-\omega}{\nu}-\alpha-\beta-\delta).
	\end{eqnarray*}
This is the well-known Lotka-Volterra equation which has been studied extensively in economic literature in particular and in mathematical literature in general. There are numerous extensions of the model since it was first introduced by Goodwin in  \cite{Goodwin} in chapter ``A Growth Cycle". See, for example, in \cite{Desai1,Desai2,Grasselli}.
	
\subsection{Keen model}
	Keen in [5] proposed a way to extend the Goodwin model by introducing new investments through the banking sector. Here again, we follow the notations from \cite{Grasselli}. Let
	$r$ be the real interest rate,
	$D$ be the debt, and 
	$d=D/Y$ be the debt ratio.
	
	The investments are captured in the nonlinear increasing function $\kappa(\pi)$, where $\pi=1-\omega-rd$ is the net profit. Here, the assumptions on $\kappa$ are as in \cite{Grasselli}:
	\begin{enumerate}[(i)]
		\item $\kappa \text{ is continuously differentiable},$
		\item $\frac{d\kappa}{d\pi} >0,$
		\item $\lim\limits_{\pi\to-\infty} \kappa(\pi)=\kappa_0 <\nu(\alpha+\beta+\delta)<\lim\limits_{\pi\to\infty} \kappa(\pi),$
		\item $\lim\limits_{\pi\to -\infty}\pi^2\kappa'(\pi)=0.$
	\end{enumerate}
	
We also require
	\begin{eqnarray}
		\Phi'(\lambda)&>&0, \\
		\Phi(0)&<&\alpha \label{phi_less}.
	\end{eqnarray}
	Integrating these assumptions into the Goodwin model leads to the following system of first order nonlinear differential equation:
	\begin{eqnarray}
		\dot{\omega}&=&\omega(\Phi(\lambda)-\alpha),			\label{line1}	\\
		\dot{\lambda}&=& \lambda(\frac{\kappa(\pi)}{\nu}-\alpha-\beta-\delta),	\label{line2}\\
		\dot{d}&=&d(r-\frac{\kappa(\pi)}{\nu}+\delta)+\kappa(\pi)-(1-\omega),	\label{line3}
	\end{eqnarray}
	where $\pi=1-\omega-rd$.

	Our analysis is taken in the framework of the Keen model.

\section{Analysis}
\subsection{Equilibria}

As it is shown in \cite{Grasselli}, there are four equilibria for the system. The first one is
\begin{equation}
	(0,0,\overline{d_0})	\label{sol1},
\end{equation}
where $\overline{d_0}$ is a solution of the equation
\begin{equation}
	d(r-\frac{\kappa(1-rd)}{\nu}+\delta)+\kappa(1-rd)-1=0 \label{eq1}.
\end{equation}
The second equilibrium is
\begin{equation}
	(\overline{\omega_1},\overline{\lambda_1},\overline{d_1}) \label{sol2},
\end{equation}
	where
	\begin{eqnarray*}
	\overline{\omega_1}&=&1-\kappa^{-1}(\nu(\alpha+\beta+\delta))-r\frac{\nu(\alpha+\beta+\delta)-\kappa^{-1}(\nu(\alpha+\beta+\delta))}{\alpha+\beta},\\
	\overline{\lambda_1}&=&\Phi^{-1}(\alpha),\\
	\overline{d_1}&=&\frac{\nu(\alpha+\beta+\delta)-\kappa^{-1}(\nu(\alpha+\beta+\delta))}{\alpha+\beta}.
	\end{eqnarray*}	
The third equilibrium is
\begin{equation}
	(0, \lambda, \overline{d_1}),	\label{sol3}
\end{equation}
where $\lambda \in (0,1)$, $1-r\overline{d_1}=\kappa^{-1}(\nu(\alpha+\beta+\delta))$, and $\overline{d_1}$ is the same as in (\ref{sol2}).

Finally, in case of explosive debt, we have the forth equilibrium
\begin{equation}
	(0,0,+\infty).	\label{sol4}
\end{equation}

Equilibrium (\ref{sol3}) is structurely unstable. The stability of (\ref{sol2}) and (\ref{sol4}) depend on the parameters and are studied in detail in \cite{Grasselli}.

In this paper, we focus on the equilibrium (\ref{sol1}). 

\subsection{Negative debts}
To study the stability of equilibrium (\ref{eq1}), we linearize the system (\ref{line1})-(\ref{line3}) about $(0,0,\overline{d_0})$. The corresponding Jacobian matrix is
\begin{equation}\label{jacobian}
	J(0,0,\overline{d_0})=\begin{bmatrix}
		\Phi(0)-\alpha	&	0	&	0\\
		0			&	\frac{\kappa(\overline{\pi_0})-\nu(\alpha+\beta+\delta)}{\nu}	&0\\
		\frac{(\overline{d_0}-\nu)\kappa'(\overline{\pi_0})+\nu}{\nu} &	0	&\frac{\nu(r+\delta)-\kappa(\overline\pi_0)+r(\overline d_0-\nu)\kappa'(\overline\pi_0)}{\nu}
	\end{bmatrix}.
\end{equation}
This matrix has three eigenvalues
	\begin{eqnarray}
		\Phi(0)-\alpha,\\
		\frac{\kappa(\overline{\pi_0})-\nu(\alpha+\beta+\delta)}{\nu}	\label{ei_1},\\
		\frac{\nu(r+\delta)-\kappa(\overline\pi_0)+r(\overline d_0-\nu)\kappa'(\overline\pi_0)}{\nu} \label{ei_2}.
	\end{eqnarray}

Note that, by assumption (\ref{phi_less}), the first eigenvalue is strictly negative.
\begin{proposition}
	Suppose $\K(x)$ has the form $c+\K_1e^{\K_2 x}$. There is a regime of parameters for the system (\ref{line1})-(\ref{line3}) such that the equilibrium $(0, 0, \overline{d_0})$ is stable and $\overline{d_0} <0$.
\end{proposition}
\begin{proof}
To prove this statement, we find sufficient conditions for the eigenvalues (\ref{ei_1}) and (\ref{ei_2}) to be negative. The following conditions must hold
	\begin{subequations}
		\begin{align}
			\K (\overline{\pi_0})-\nu(\alpha+\beta+\delta)	&<	0	\label{cond_1a},\\
			\nu(r+\delta)-\K(\overline\pi_0)+r(\overline d_0-\nu)\K'(\overline\pi_0)	&<	0	\label{cond_1b},\\
			d[r-{\K (1-rd)\over \nu}+\delta]+\K (1-rd)-1	&=	0.	\label{cond_1c}
		\end{align}
	\end{subequations}
	We furthermore assume that the following holds
	\begin{subequations}
		\begin{align}
			\nu(\alpha+\beta+\delta)	<	1	\label{as_1},	\\
			r	<	\alpha+\beta			\label{as_2}.
		\end{align}
	\end{subequations}
	Set $ \K(x)= c+\K_1e^{\K_2x} $. The goal is to find the conditions for $c$ and $ \overline{d_0} $ and construct a function $ \K(x) $ that satisfies (\ref{cond_1a})-(\ref{cond_1c}).
	The equation (\ref{cond_1c}) implies that
	\begin{equation}
		\K(1-rd)={1-d(r+\delta) \over 1- d/\nu}.		\label{eq_2}
	\end{equation}
	We substitute (\ref{eq_2}) into (\ref{cond_1a}) to find that
	\begin{equation}
		\overline d_0	<	{\nu(\alpha+\beta+\delta)-1 \over \alpha+\beta-r}	<0 \label{d_bar}.
	\end{equation}
	We can now pick any \( \overline d_0 \in (-\infty; {\nu(\alpha+\beta+\delta)-1 \over \alpha+\beta-r}) \). Once we have picked \( \overline d_0 \), we can determine \( \K(\overline \pi_0) \).

	The inequality (\ref{cond_1b}) is equivalent to
	\begin{equation}
		r(\overline d_0-v)\K_2k_0(\overline\pi_0)<-\nu(r+\delta)+c+k_0(\overline\delta_0),	\label{eq_1}
	\end{equation}
	where \( k_0(x)=\K_1e^{\K_2x} \). From (\ref{eq_2}), 
	\begin{equation}
		k_0(1-rd)={1-d(r+\delta) \over 1- d/\nu}-c={1-c-d(r+\delta-c/\nu) \over 1-d/\nu}.	\label{eq_3}
	\end{equation}
	Next, we choose $c$ such that \( c < \nu(r+\delta) \). The assumption on \( c \) leads to \( k_0(\overline\pi_0) >0 \). Since \( \overline d_0 <0 \), (\ref{eq_1}) becomes
	\begin{equation}
		\K_2>{c-\nu(r+\delta)+k_0(\overline\pi_0) \over k_0(\overline\pi_0)r(\overline d_0 -\nu)}.\label{bound}
	\end{equation}
	We also require that $\kappa_2>0$. After picking \( \K_2 \in (\max\{0,{c-\nu(r+\delta)+k_0(\overline\pi_0) \over k_0(\overline\pi_0)r(\overline d_0 -\nu)}\},+\infty)\), one can compute \( \K_1 \) to be
	\begin{equation}
		\K_1={k_0(\overline \pi_0) \over e^{\K_2\overline\pi_0}}.
	\end{equation}
The function $\kappa(x)=c+\kappa_1e^{\kappa_2x}$ with $c$, $\kappa_1$ and $\kappa_2$ chosen as described above satisfies (\ref{cond_1a})-(\ref{cond_1c}). We have successfully constructed the function \( \K(x) \) such that all of the eigenvalues of the Jacobian (\ref{jacobian}) are negative and therefore the equilibrium $(0,0,\overline{d_0})$ is stable.

Furthermore, by (\ref{d_bar}), $\overline{d_0}<0$ as desired.
\end{proof}

Here we construct an example of the function $\kappa(x)$ using the method we described.
\begin{example}
	Choose  $r=0.03$, $\alpha=0.05$, $\beta= 0.03$, $\delta=0.1$, $\nu=3$ and
	$$\kappa(x)= 0.34 + 0.0836e^{0.6829x}.$$
	Take the Philips curve to be as in \cite{Grasselli}:
	\begin{equation}
		\Phi(\lambda)=\frac{\phi_1}{(1-\lambda)^2}-\phi_0,
	\end{equation}
	where
	\begin{equation}
		\phi_0=\frac{0.04}{1-0.04^2}, \phi_1=\frac{0.04^3}{1-0.04^2}.	
	\end{equation}
	The equation (\ref{cond_1c}) has two solutions
	\begin{equation}
		\overline{d_0}\approx\begin{cases}
			-9.2100,\\
			86.5545.
		\end{cases}
	\end{equation}
	We are intersted in the negative real root of (\ref{cond_1c}). The eigenvalues at $(0,0,-9.2100)$ are
	\begin{equation}
		(-0.0900, -0.00012285, -0.1664).
	\end{equation}
So, $(0,0,-9.2100)$ is a stable equilbirum.
\end{example}

In general, one cannot predict the behavior of the system in the neighborhood  of the equilibrium $(0,0,\overline{d_0})$ unless additional information is avalable. In the literature, it is assumeed that in most cases the two eigenvalues (\ref{ei_1}) and (\ref{ei_2})   have opposite signs. However, it seems likely that there exist  regions in the parameter space  where the eigenvalues are of the same sign, possibly both negative. In any case, it is clear that more investigation of this matter is needed. A possible starting point is to describe the regimes where the two eigenvalues (\ref{ei_1}) and (\ref{ei_2}) are both equal to zero.

\begin{proposition}
	Given $\kappa(x)=c+\kappa_1e^{\kappa_2x}$, a necessary condition for (\ref{ei_1}) and (\ref{ei_2}) to be both 0 is that
	$$c\ge 2\sqrt{\frac{(r+\delta)(1-2\nu(r+\delta))}{rk_2}}+1-\nu(r+\delta),$$
	or
	$$c\le -2\sqrt{\frac{(r+\delta)(1-2\nu(r+\delta))}{rk_2}}+1-\nu(r+\delta).$$
\end{proposition}

\begin{proof}
	The proof of this condition is a straight forward algebraic calculation.
	We require that
	\begin{subequations}
		\begin{align}
			\K (\overline{\pi_0})-\nu(\alpha+\beta+\delta)	&=	0,	\label{cond_2a}\\
			\nu(r+\delta)-\K(\overline\pi_0)+r(\overline d_0-\nu)\K'(\overline\pi_0)	&=	0	\label{cond_2b},\\
			d[r-{\K (1-rd)\over \nu}+\delta]+\K (1-rd)-1	&=	0	\label{cond_2c}.
		\end{align}
	\end{subequations}
	We substitute (\ref{cond_2a}) into (\ref{cond_2c}) and obtain
	\begin{equation}
		d=\frac{1-\nu(\alpha+\beta+\delta)}{r+\delta}.
	\end{equation}
	The equation (\ref{cond_2b}) then becomes
	\begin{equation}
		\nu(r+\delta)-\nu(\alpha+\beta+\delta)+r(\frac{1-\nu(\alpha+\beta+\delta)}{r+\delta}-\nu)(\nu(\alpha+\beta+\delta)-c)\kappa_2=0.
	\end{equation}
	For clarity, let $A=\nu(r+\delta)$ and $B=\nu(\alpha+\beta+\delta)$. We then have
	\begin{equation}
			A-B+r\left(\frac{1-A-B}{r+\delta}\right)(B-c)\kappa_2=0.
	\end{equation}
	This is equivalent to
	\begin{equation}
		-\frac{rB^2\kappa_2}{r+\delta}+B(-1+\frac{r(1-A)\kappa_2}{r+\delta}+\frac{rc\kappa_2}{r+\delta})+A-\frac{r(1-A)c\kappa_2}{r+\delta} =0.		\label{del}
	\end{equation}
	We require (\ref{del}) to have real roots. This means that
	\begin{equation}
		\left(-1+\frac{r(1-A)\kappa_2}{r+\delta}+\frac{rc\kappa_2}{r+\delta}\right)^2-4\frac{r^2(1-2A)c\kappa_2^2}{(r+\delta)^2} \ge 0,
	\end{equation}
	which is equivalent to
	\begin{equation}
		\left(\frac{r\kappa_2}{r+\delta}(1-A-c)+1\right)^2\ge \frac{4r\kappa_2}{r+\delta}(1-2A).
	\end{equation}
	If $1-A\le 0$, then this is always true. If $1-A\ge 0$ then we will have
	\begin{equation}
		\left|\frac{r\kappa_2}{r+\delta}(1-A-c)\right|\ge \sqrt{\frac{4r\kappa_2(1-2A)}{r+\delta}}.
	\end{equation}
	This completes the proof.
\end{proof}

\section{Discussions}
The equilibrium $(0,0,\overline{d_0})$ with $\overline{d_0}<0$ represents the economic state when there is no wage, no employment, but the debt is negative. In the liturature, due to its likelihood of being unstable, this equilibrium is believed to be economically meaningless. However, the authors can think of a situation when this is not at all meaningless: supose there is a break down in the economy due to war. There is no production. Yet, in the country, there is still a lot of products that are consumable. In order to have food, person A, who does not have money, has to borrow money from the government, who has the power of printing money and issuing credits. This continues for a long time for everyone. Over time, negative debts accumulate for the government. In this situation, the economy is still somewhat functional, as long as the consumable proucts do not run out.

Furthermore, while the equilibrium studied in this paper may not be of any interest by itself, a good understanding of the equilibirum will lead to a better understading of the dynamics of the model as a whole. A series of vital questions then arise: 1) When would this equilibrium turn from stable to unstable? 2) Is stability the necessary condition when one has negative debt in the economy? 3) Is there a limit cycle that arises around the equilibirum when it changes from stable to unstable? If yes, is the limit cycle stable or unstable?

\section{acknowledgements}
We thank Dr. Anna Ghazaryan for her encouragement and help and invaluable discussions. She posed the problem, gave us directions and helped us with the writing.

Teng Zhang was in part supported by the NSF grant DMS-1311313 to Anna Ghazaryan.

\bibliography{bibo}
\bibliographystyle{plain}
\end{document}